\documentclass{amsart}
\usepackage[margin=1in]{geometry}
\usepackage{amssymb,amsfonts,amsmath,amsthm,enumitem}
\usepackage[hidelinks]{hyperref}

\DeclareMathOperator{\tr}{tr}
\DeclareMathOperator{\spt}{spt}
\DeclareMathOperator{\dvg}{div}
\DeclareMathOperator{\Hess}{Hess}
\DeclareMathOperator{\Reg}{Reg}
\DeclareMathOperator{\Sing}{Sing}

\newcommand{\norm}[1]{\left\lVert#1\right\rVert}
\newcommand{\ip}[2]{\left\langle#1,#2\right\rangle}

\newcommand{\R}{\mathbb{R}}
\newcommand{\Mass}{\mathbf{M}}
\newcommand{\I}{\mathbf{I}}
\newcommand{\mres}{\mathbin{\vrule height 1.6ex depth 0pt width 0.13ex\vrule height 0.13ex depth 0pt width 1.1ex}}
\newcommand{\llbracket}{\mathopen{[\mkern-3mu[}}
\newcommand{\rrbracket}{\mathclose{]\mkern-3mu]}}

\theoremstyle{plain}
\newtheorem{theorem}{Theorem}[section]
\newtheorem{proposition}[theorem]{Proposition}
\newtheorem{lemma}[theorem]{Lemma}
\newtheorem{corollary}[theorem]{Corollary}
\theoremstyle{definition}
\newtheorem{definition}[theorem]{Definition}

\theoremstyle{remark}
\newtheorem{remark}[theorem]{Remark}

\allowdisplaybreaks

\begin{document}

\author{Aidan F. Wood}
\address{Department of Mathematics, University of Connecticut, Storrs, CT 06269} \email{aidan.wood@uconn.edu}

\title[Convex hulls and minimizing currents in ALH manifolds]{Convex hulls and area minimizing currents in asymptotically locally hyperbolic manifolds}
\date{}

\begin{abstract}
We solve the oriented asymptotic Plateau problem for integral currents of arbitrary codimension in conformally compact asymptotically locally hyperbolic manifolds without global curvature assumptions. This extends Anderson's theorem in hyperbolic space. A new ingredient is the construction of a barrier hull near conformal infinity, which serves as a barrier for area minimizing currents and stationary varifolds. In the conformally compact setting, the associated convex hull identity extends Anderson's result beyond globally pinched negative curvature.
\end{abstract}

\maketitle

\section{Introduction}

Let $(M^{n+1},g)$ be a $C^2$ conformally compact asymptotically locally hyperbolic Riemannian manifold. Then there is a smooth defining function $\rho$ such that $\rho^2g$ extends to a nondegenerate $C^2$ metric on $\overline{M}$, with $|d\rho|_{\rho^2g}=1$ on $\partial M$. This terminology is motivated by a quick computation with the conformal curvature formula which shows that the sectional curvatures of $g$ tend to $-1$ at the boundary \cite{Mazzeo1988}. No global sign or pinching condition on the sectional curvature is assumed, and the conformal boundary does not have to be spherical. See Remark \ref{rem:ah-alh-terminology} for a comparison with related definitions.

We begin with convexity at conformal infinity. A subset of $M$ is \emph{totally convex} if it contains every geodesic segment whose endpoints lie in the subset. Our first construction finds functions with positive Hessian on large enough positive superlevel sets and which satisfy the static equation to leading order there. The complement of the superlevel set of such functions is totally convex by a direct geodesic argument in Proposition \ref{prop:totally-convex-complements}. This gives the first result.

\begin{theorem}\label{thm:barrier-hull}
    Let $(M^{n+1}, g)$ be $C^2$ conformally compact asymptotically locally hyperbolic. For every closed $S\subseteq \partial M$, there exists a closed totally convex set $\mathcal{H}(S) \subseteq M$ such that
    \[
        \partial_\infty \mathcal{H}(S) = S.
    \]
    If $V$ is a stationary rectifiable $p$-varifold in $M$, for $1 \leq p \leq n$, and $\partial_\infty \spt V \subseteq S$, then
    \[
        \spt V \subseteq \mathcal{H}(S).
    \]
\end{theorem}

To distinguish it from other hulls, we call $\mathcal{H}(S)$ the \emph{barrier hull}. It depends on the choice of defining function and a specified family of convex sets. The barrier hull is intentionally different from the familiar convex hull obtained by intersecting the closures of all closed totally convex sets whose asymptotic boundaries contain $S$. Note, though, that since the barrier hull itself is one such set and has asymptotic boundary $S$, this convex hull also meets the conformal boundary exactly in $S$. Gromov suggested that the corresponding identity for convex hulls in hyperbolic space should hold in general ``strictly convex spaces'' \cite{Gromov1981}. Corollary~\ref{cor:gromov-hull} establishes that identity in the present setting.

For the Plateau problem, we regard $\overline{M}$ as a compact domain in a smooth closed ambient manifold, so that cycles on $\partial M$ and current boundary identities are understood in the usual sense. Assuming $C^3$ regularity in the following theorem ensures that the geodesic compactification is $C^2$, as required for the argument.

\begin{theorem}\label{thm:asymptotic-plateau}
    Let $(M^{n+1}, g)$ be $C^3$ conformally compact asymptotically locally hyperbolic, and let $2 \leq p \leq n$. Suppose that $A\in\I_{p-1}(\partial M)$ is a nonzero integral cycle and that $\partial Q=A$ for some $Q\in\I_p(\overline{M})$. Then there exists $\overline{T} \in \I_p(\overline{M})$ such that $\partial \overline{T}=A$ and $\overline{T}$ has no mass on $\partial M$.
    If $T := \overline{T} \mres M$, then
    \begin{enumerate}[label=\textup{(\alph*)}]
        \item $\partial T = 0$,
        \item $T$ is locally absolutely area minimizing in $(M,g)$,
        \item $\partial_\infty \spt T = \spt A$ and $T \neq 0$.
    \end{enumerate}
\end{theorem}

In particular, although much more general prescribed boundary is permitted, $A$ may be taken to be a smooth closed oriented immersed submanifold of $\partial M$. Theorem \ref{thm:asymptotic-plateau} solves the oriented asymptotic Plateau problem for every nonzero integral cycle on $\partial M$ which bounds in the compactification. In codimension one, standard regularity theory provides the following consequences.

\begin{corollary}\label{cor:codim-one}
    In the setting of Theorem \ref{thm:asymptotic-plateau}, assume further that $p=n$. Write $\Reg T$ for the regular part of $\spt T$ and $\Sing T:=\spt T\setminus\Reg T$. Then $\Reg T$ is a smooth embedded minimal hypersurface, and the multiplicity is a locally constant positive integer on each connected component of $\Reg T$. Furthermore,
    \begin{enumerate}[label=\textup{(\alph*)}]
        \item if $n \leq 6$, then $\spt T$ is a complete properly embedded smooth minimal hypersurface in $(M,g)$;
        \item if $n=7$, then $\Sing T$ is locally finite;
        \item if $n \geq 8$, then $\Sing T$ has Hausdorff dimension at most $n-7$.
    \end{enumerate}
\end{corollary}

In the hypersurface case, Theorem \ref{thm:asymptotic-plateau} supplies a rigorous existence theorem and a cutoff limit along a subsequence for an anchored Plateau problem closely related to the minimizing step on a fixed slice in maximin and quantum maximin constructions of holographic entanglement entropy in AdS/CFT. See Remark \ref{rem:fixed-slice}.

Minimal submanifolds with prescribed conformal boundary are also needed in the Poincar\'e--Einstein renormalized area program of Graham--Witten \cite{GrahamWitten1999}. Graham--Reichert study formal asymptotics for submanifolds which are minimal to high order at infinity \cite{GrahamReichert2020}, while the recent Gauss--Bonnet formula of Case--Graham--Kuo--Tyrrell--Waldron assumes a polyhomogeneous minimal immersion \cite{CaseGrahamKuoTyrrellWaldron2025}. Theorem \ref{thm:asymptotic-plateau} complements this theory by providing an existence theorem for the underlying oriented Plateau problem without assuming Einstein.

Anderson's seminal 1982 theorem solved the asymptotic Plateau problem in hyperbolic space \cite{Anderson1982}. His theorem is significant both in its generality, for it allows boundary data in arbitrary dimension and codimension, and in its contrast with Bernstein-type results in Euclidean space. His proof uses radial flow to large geodesic spheres, cone competitors whose mass can be computed explicitly, geodesic ball monotonicity, and totally geodesic half-spaces as barriers.

In subsequent work, Anderson introduced a very elegant scalloping construction in complete simply connected manifolds with pinched negative sectional curvature. It produces global convex sets with controlled asymptotic boundary and yields the convex hull identity for arbitrary closed subsets of the asymptotic boundary, which in hyperbolic space follows directly from totally geodesic half-spaces \cite{Anderson1983Dirichlet}. Jacobi field comparison gives the local convexity used to construct the scallops, while Rauch comparison gives the angular control needed to identify the asymptotic boundary of the resulting convex set. Heuristically, negatively curved space is \emph{so} convex that one can remove chunks from geodesic balls in the right way and the result remains totally convex. Both arguments leverage the global distance functions and geodesic convexity provided by constant and pinched negative sectional curvature.

There are broader existence results under other global assumptions. Lang solved the asymptotic Plateau problem for a large class of limit sets in Gromov hyperbolic Hadamard manifolds of bounded geometry \cite{Lang2003}. Casteras--Holopainen--Ripoll obtained results on Cartan--Hadamard manifolds under the strict convexity condition at infinity, together with curvature or coefficient hypotheses \cite{CHR2018}. We articulate a more detailed comparison with those results at the end of Section \ref{sec:convex-sets}. Anderson's constructions, and these other strategies, do not transfer directly to the present setting.

Section \ref{sec:preliminaries} fixes collar coordinates and GMT conventions. Section \ref{sec:convex-sets} proves the barrier and convex hull results and compares them with earlier notions of convexity at infinity. Sections \ref{sec:compact-plateau-and-mass-estimates} and \ref{sec:passage-to-limit} carry out the Plateau construction.

\medskip

\noindent\textbf{Acknowledgments.} The author was partially supported by the National Science Foundation under Grant No.\ DMS-2304966, and is grateful to Lan-Hsuan Huang for helpful discussions. 

\section{Preliminaries}\label{sec:preliminaries}

Let $(M^{n+1}, g)$ be $C^2$ conformally compact, so there exists a smooth defining function $\rho$ such that $\rho^2g$ extends to a nondegenerate $C^2$ metric for $\overline{M}$. Throughout, $M$ and $g$ are smooth in the interior. The asymptotic boundary of $A\subset M$ is
\[
    \partial_\infty A:=\overline{A}^{\,\overline{M}}\cap\partial M,
\]
where $\overline{A}^{\,\overline{M}}$ denotes closure in $\overline{M}$. We say that $(M,g)$ is \emph{asymptotically locally hyperbolic} if $|d\rho|_{\rho^2g} = 1$ on $\partial M$. This condition is independent of the choice of defining function.

The use of \emph{hyperbolic} in the name can be explained by the conformal curvature formula. In coordinates smooth up to the boundary, it gives
\[
    R_{ijkl} = -|d\rho|^2_{\rho^2g}(g_{ik}g_{jl} - g_{il}g_{jk}) + O(\rho^{-3}).
\]

It is important to note that this is only an asymptotic curvature condition, and that no global sign or pinching condition on the sectional curvature is assumed. In particular, the curvature of the compact part is unrestricted, and the conformal boundary need not be spherical. The present work also does not assume the Einstein equation or a polyhomogeneous expansion.

Poincar\'e--Einstein metrics, which satisfy $\operatorname{Ric}_g=-ng$, do form an important subclass. Graham and Lee constructed such metrics on the ball with conformal infinity close to the round conformal class \cite{GrahamLee1991}. Another standard family consists of the Birmingham--Kottler metrics \cite[Example~4]{ChruscielGalloway2021}. If $(N,h_k)$ has constant sectional curvature $k$ and is closed, its end metric
\[
    g_m=\frac{dR^2}{R^2+k-2mR^{1-n}}+R^2h_k
\]
is another valid model wherever the denominator is positive. With $\rho=R^{-1}$, its compactification is
\[
    \rho^2g_m=\frac{d\rho^2}{1+k\rho^2-2m\rho^{n+1}}+h_k,
\]
so the defining condition is immediate.

\begin{remark}\label{rem:ah-alh-terminology}
Terminology for asymptotically hyperbolic and asymptotically locally hyperbolic varies across the literature. Modulo regularity, the condition imposed here is the one Graham--Lee and Chru{\'s}ciel--Herzlich call \emph{asymptotically hyperbolic} \cite{GrahamLee1991, ChruscielHerzlich2003}. Andersson--Dahl use \emph{asymptotically locally hyperbolic} for ends strongly asymptotic to finite quotients of hyperbolic space and reserve \emph{asymptotically hyperbolic} for trivial quotients~\cite{AnderssonDahl1998}. Chru{\'s}ciel--Galloway call a smooth conformally compact metric \emph{asymptotically locally hyperbolic} when all its sectional curvatures tend to $-1$ at conformal infinity, and call a component of its conformal boundary at infinity, or the corresponding end, \emph{asymptotically hyperbolic} when the conformal metric on that component is conformal to a round sphere \cite{ChruscielGalloway2021}. \v{C}ap--Gover use the boundary condition imposed here as their definition of \emph{asymptotically locally hyperbolic} and note both conventions \cite[Remark~2.3]{CapGover2024}. Lee--Neves use \emph{asymptotically locally hyperbolic} for a different class in three dimensions. They require a constant curvature metric at conformal infinity, an asymptotic expansion defining a mass aspect function, and quantitative decay, and use \emph{asymptotically hyperbolic} when conformal infinity is the round sphere \cite{LeeNeves2015}.
\end{remark}

If $h=\rho^2g$ is a compactification, then $h|_{T\partial M}$ depends on the defining function, but its conformal class does not. We say unambiguously, then, that this class is the conformal infinity of $(M,g)$. For the proof of Theorem \ref{thm:asymptotic-plateau}, we take $h$ to be $C^3$ and set $h_0:=h|_{T\partial M}$. The geodesic defining function $x$ is uniquely determined in a collar by
\[
    |dx|_{x^2g}=1, \qquad
    (x^2g)|_{T\partial M}=h_0.
\]
The Graham--Lee normal form gives a unique such function in a sufficiently small collar \cite[Lemma~5.2]{GrahamLee1991}. Since $h$ is $C^3$, the geodesic compactification associated with $h$ is $C^2$ up to the boundary \cite[Section~1]{Anderson2003Boundary}. Choose $\varepsilon_0>0$ so that the geodesic normal form is defined on the collar $\{0<x<2\varepsilon_0\}$. Extend $x$ to a positive Lipschitz function on $M$ without changing it on this collar, and so that $x\geq2\varepsilon_0$ off the collar. In geodesic normal coordinates $(x, y^1,\ldots,y^n)$, we can write
\[
    \overline{g} := x^2g = dx^2+h_x, \qquad
    g = \frac{dx^2 + h_x}{x^2}
\]
in the collar neighborhood, where $h_x$ is a family of metrics on the boundary with $h_x|_{x=0} = h_0$. Setting $r=-\log x$, we have, in the collar,
\[
    g = dr^2 + e^{2r}g_{r},
\]
with $g_r := h_{e^{-r}}$. With this extension, $r$ is a proper Lipschitz function on all of $M$. Let $r_0 := -\log{\varepsilon_0}$. For $R\geq r_0$, set $O_R:=\{r\leq R\}$ and $N_R:=\partial O_R=\{r=R\}$. Each $O_R$ is a \emph{compact core} of $M$. Equivalently, a compact core is the complement of a collar neighborhood of conformal infinity.

We use the two defining functions for different purposes. The smooth function $\rho$ and the compactification $h=\rho^2g$ will be used for the barrier construction. The variables $x$ and $r=-\log x$ will be used for radial retraction, slicing, and the mass estimates.

For large enough $r_0$, we have uniformly on the collar that
\[
    |g_r - h_0| \leq Ce^{-r}, \qquad
    \partial_r g_r = -e^{-r}\,(\partial_x h_x)\big|_{x = e^{-r}} = O(e^{-r}),
\]
where $C$ depends only on $(\overline{M}, \overline{g}, x)$. In particular, we can have $\frac{1}{2}h_0 \leq g_r \leq 2h_0$ for large enough $r_0$.

On every finite annulus bounded away from $x=0$, the normal form coordinate $r$ is $C^2$. Each finite level $N_R$ is a $C^2$ hypersurface, $O_R$ is a compact Lipschitz neighborhood retract, and the fields $w(r)\nabla r$ used below are $C^1$ wherever they are supported.

Using the convention $\Delta_g f=\tr_g \Hess_g f$, we say that a function $f\in C^2_{\mathrm{loc}}(M)$, not identically zero, is a \emph{static potential} if
\[
    L_g^* f=-(\Delta_g f)g+\Hess_g f-f\operatorname{Ric}_g=0,
\]
where $L_g^*$ is the formal $L^2$-adjoint of the linearized scalar curvature operator \cite[Section~3]{HuangJangMartin2020}. If $\operatorname{Ric}_g=-ng$, the static equation gives $\Hess_g f=fg$.

For geometric measure theory conventions and basic facts about currents and rectifiable varifolds, we follow the notes of Simon \cite{Simon2016}, with the usual passage via charts from Euclidean open sets to smooth Riemannian manifolds. Let $N$ be a smooth Riemannian manifold without boundary. We write $\I_{k,\mathrm{loc}}(N)$ and $\I_k(N)$ for the locally integral and integral $k$-currents, respectively, $\norm{T}_g$ for the mass measure of a current, and $\Mass_g(T)$ for its total mass. All currents have coefficients in $\mathbb{Z}$, and $T\mres A$ denotes restriction to a Borel set $A$.

For a rectifiable $k$-varifold $V$, we write $\mu_V$ for its weight measure, $T_xV$ for its approximate tangent plane, and $\spt V:=\spt\mu_V$. If $X\in C_c^1(N;TN)$ and $P\subset T_xN$ is a $k$-plane with orthonormal basis $(e_1,\ldots,e_k)$, set
\[
    \dvg_P X:=\sum_{i=1}^k\ip{\nabla_{e_i}X}{e_i}.
\]
The first variation of $V$ is
\[
    \delta V(X):=\int\dvg_{T_xV} X\,d\mu_V(x).
\]
The varifold is \emph{stationary} in an open set $U\subset N$ if $\delta V(X)=0$ for every $X\in C_c^1(U;TN)$. For an integer multiplicity rectifiable current $T$, its associated varifold $|T|$ drops orientation but satisfies
\[
    \mu_{|T|}=\norm{T}_g, \qquad
    \spt|T|=\spt T.
\]

We say that $T\in\I_{k,\mathrm{loc}}(N)$ is \emph{locally absolutely area minimizing} if, whenever $S\in\I_{k,\mathrm{loc}}(N)$ satisfies $\partial S=\partial T$ and $S-T$ has compact support, one has
\[
    \norm{T}_g(W)\leq\norm{S}_g(W)
\]
for every relatively compact open set $W$ containing $\spt(S-T)$. If $T$ is locally absolutely area minimizing in an open set $U$ and $U\cap\spt\partial T=\emptyset$, then $|T|$ is stationary in $U$.

Now let $k\geq2$, $T\in\I_{k,\mathrm{loc}}(N)$, and let $f$ be locally Lipschitz. For a.e.\ $s\in\R$, the slice
\[
    \langle T,f,s\rangle :=\partial\bigl(T\mres\{f<s\}\bigr)-(\partial T)\mres\{f<s\}
\]
is a locally integral $(k-1)$-current. Write $\nabla^Tf$ for the approximate tangential gradient of $f$. For every open set $W\subset N$ and every $a<b$, the integrated slicing identity is
\[
    \int_a^b\norm{\langle T,f,s \rangle}_g(W)\,ds =\int_{W\cap\{ a < f < b\}}|\nabla^Tf|_g\,d\norm{T}_g.
\]

If $F:(N,g)\to(N',g')$ is locally Lipschitz and proper on $\spt T$, then
\[
    \partial(F_\#T)=F_\#(\partial T), \qquad
    \spt(F_\#T)\subseteq F(\spt T).
\]
If, in addition, $F$ is $L$-Lipschitz on $\spt T$, then
\[
    \Mass_{g'}(F_\#T)\leq L^k\Mass_g(T).
\]

To consider current boundaries on $\partial M$, fix a smooth closed manifold $\widetilde{M}$ containing $\overline{M}$ as a compact domain, together with a smooth metric $\widetilde{g}$ which is uniformly equivalent to $\overline{g}$ on $\overline{M}$. We identify currents on $\partial M$ with their pushforwards under inclusion, write $\I_k(\overline{M})$ for the integral currents in $\widetilde{M}$ supported in $\overline{M}$, and denote the corresponding ambient flat norm by $\mathcal{F}_{\widetilde{g}}$. For smooth submanifold boundary data, Anderson similarly regards the ball model as a domain in $\R^n$ and establishes the corresponding boundary current identity for the minimizing current constructed in his proof \cite[Proposition~6]{Anderson1982}.

Finally, as measures on the interior collar $\{0<x<\varepsilon_0\}$, conformal scaling gives
\[
    d\norm{Z}_{\overline{g}}=x^k\,d\norm{Z}_g=e^{-kr}\,d\norm{Z}_g
\]
for every rectifiable current $Z$ of dimension $k$ in $M$.

\section{Barrier hulls and convexity at conformal infinity}\label{sec:convex-sets}

In hyperbolic space, totally geodesic half-spaces separate arbitrary closed subsets of the boundary sphere. Anderson's scalloping gives analogous separation at infinity under pinched negative curvature \cite{Anderson1983Dirichlet}. Our construction proceeds in the opposite direction: we choose a function $\psi$ on $\partial M$ so that $\overline{\{\psi>0\}}$ prescribes the asymptotic boundary of the superlevel region, extend it along a collar in $\overline{M}$, and take a sufficiently large level of $u_\psi=\phi_\psi/\rho$. Strict convexity of this function on the resulting superlevel region makes the complement totally convex.

We begin with the computation in hyperbolic space, where all sectional curvatures equal $-1$. We use the upper half-space model,
\[
    \mathbb{H}^{n+1} = \{(z,y): z > 0, y \in \R^n\}, \qquad
    g_H = z^{-2}(dz^2 + |dy|^2).
\]
Let
\[
    u_\psi(z,y)=\frac{\psi(y)}{z}.
\]
In this metric, we can compute the nonzero Christoffel symbols to be
\[
    \Gamma_{zz}^z = -z^{-1},\qquad
    \Gamma_{ij}^z = z^{-1}\delta_{ij}, \qquad
    \Gamma_{zj}^i = -z^{-1}\delta^i_j.
\]
Since
\[
    \partial_z u_\psi = -\frac{\psi}{z^2}, \qquad
    \partial_i u_\psi = \frac{\psi_i}{z},
\]
substituting into the coordinate formula for the Hessian $u_{;ab} = \partial_a \partial_b u - \Gamma^c_{ab}\partial_c u$ gives
\[
    (u_{\psi})_{;zz} = \frac{\psi}{z^3} = u_\psi (g_{H})_{zz}, \qquad
    (u_\psi)_{;zi} = 0,\qquad
    (u_{\psi})_{;ij} = u_{\psi}(g_H)_{ij} + \frac{\psi_{ij}}{z}.
\]
If we specify the orthonormal frame $e_0 = -z\partial_z$ and $e_i = z\partial_{y_i}$, we get
\[
    \Hess_H u_\psi(e_0, e_0) = u_\psi, \qquad
    \Hess_H u_\psi(e_0,e_i) = 0, \qquad
    \Hess_H u_{\psi}(e_i,e_j) = u_\psi \delta_{ij} + z \psi_{ij}.
\]
When $\psi$ is affine, its Euclidean Hessian in the boundary coordinates vanishes, so $u_\psi$ is a static potential. For general $\psi$, the error in the static equation is $zD^2\psi$ in the tangential orthonormal directions, so
\[
    |\Hess_H u_\psi - u_\psi g_H|_{g_H} = z|D^2\psi|_{\delta}.
\]
The factor $z$ makes this error small on sufficiently large positive level sets and yields the following threshold. Define
\[
    Q_\psi := \sup_{\{\psi > 0\}} \psi(-\lambda_{\min}(D^2\psi))_+.
\]
Assume that $Q_\psi < \infty$. On $\Omega := \{u_\psi > c\}$ we have $z < \psi / c$ by construction, so we get the estimate
\[
    z^2(-\lambda_{\min}(D^2\psi))_+ \leq \frac{Q_\psi}{c^2}\psi.
\]
Then if we let $0 \leq \kappa < 1$, and take $c$ such that
\[
    c^2 \geq \frac{Q_\psi}{1-\kappa},
\]
this implies
\[
    \Hess_H u_\psi \geq \kappa u_\psi g_H
\]
on $\Omega$. For a tangible example, take
\[
    \psi_R(y) = \frac{R^2 - |y|^2}{2}.
\]
We have $\Omega_{\psi_R,c} = \{u_{\psi_R} > c\} = \{|y|^2 + 2cz < R^2\}$. The boundary of this set is a paraboloid. If we further let
\[
    \zeta := \frac{R^2 - |y|^2 - z^2}{2z} = u_{\psi_R} - \frac{z}{2},
\]
we can see that $\zeta$ is a static potential related to $u_{\psi_R}$ by the error term $z/2$. If we choose $R = 1$ and set $E:=\Hess_H u_{\psi_1}-u_{\psi_1}g_H$, then $D^2\psi_1 = -I$, and
\[
    E_{zz} = E_{zi} = 0, \qquad
    E_{ij} = -z^{-1}\delta_{ij}, \qquad
    |E|_{g_H} = \sqrt{n} z.
\]

We now return to the smooth defining function $\rho$ fixed in the preliminaries and to the general setting of asymptotically locally hyperbolic manifolds. In the upper half-space model, one may take $\rho=x=z$, and the function constructed below is then exactly $u_\psi=\psi/z$ from the preceding computation.

Let $a_\rho=|d\rho|_h^2$. Choose $\rho_0>0$ so that $\{0\leq\rho<\rho_0\}$ is a collar on which $a_\rho>0$, and set $W=\nabla^h\rho/a_\rho$. Then $d\rho(W)=1$, and the flow of $W$ identifies this collar with $[0,\rho_0)\times\partial M$. Let $\pi_\rho$ denote the resulting projection onto $\partial M$. For $\psi\in C^2(\partial M)$ let
\[
    \phi_\psi=\psi\circ\pi_\rho, \qquad
    u_\psi=\frac{\phi_\psi}{\rho}.
\]
Then $d\phi_\psi(W)=0$. If $\psi$ is positive somewhere and $c>0$, write
\[
    \Omega_{\psi,c}:=\{u_\psi>c\}\subseteq(0,\rho_0)\times\partial M.
\]
On this collar, we have
\begin{equation}\label{eq:omega-coordinates}
    \Omega_{\psi,c}
    =\{(\rho,y):\psi(y)>0,\ 0<\rho<\psi(y)/c\}.
\end{equation}
Since $du_\psi(W)=-u_\psi/\rho$, every positive level is a $C^2$ hypersurface. Thus \eqref{eq:omega-coordinates} gives an explicit collar description of $\Omega_{\psi,c}$ and determines its asymptotic boundary. It also shows
\[
    \Omega_{\psi,c}\subseteq\{\rho<(\max\psi)/c\}, \qquad
    \pi_\rho(\Omega_{\psi,c})\subseteq\{\psi>0\}.
\]

Several explicit choices of $\psi$ may help illustrate the construction. If, say, $\psi$ is a positive constant, then $\Omega_{\psi,c}$ is a collar neighborhood of $\partial M$. If $U\subseteq\partial M$ is open and $\spt\psi\subseteq U$, then the projection of $\Omega_{\psi,c}$ to $\partial M$ is contained in $U$. 

\begin{lemma}\label{lem:almost-static-potentials}
    There is a constant $C$ such that, for every $\psi\in C^2(\partial M)$,
    \[
        \Hess_g u_\psi=u_\psi g+E, \qquad
        |E|_g\leq C\rho\bigl(\norm{\psi}_{C^2(\partial M)}+|u_\psi|\bigr).
    \]
\end{lemma}

\begin{proof}
    For any \(f\in C^2\), the conformal change \(g=\rho^{-2}h\) gives
    \[
        \Hess_g f =\Hess_h f+\rho^{-1}\bigl(d\rho\otimes df+df\otimes d\rho\bigr)-\rho^{-1}\ip{d\rho}{df}_h h.
    \]
    Taking \(f=u_\psi=\phi_\psi/\rho\), expanding \(\Hess_h(\phi_\psi/\rho)\), and subtracting \(u_\psi g\) yields
    \begin{align}
        E ={}& \rho^{-1}\Hess_h \phi_\psi - \phi_\psi\rho^{-2}\Hess_h \rho - \rho^{-2}\ip{d\rho}{d\phi_\psi}_h h + \phi_\psi\rho^{-3}(a_\rho - 1)h. \label{eq:static-error}
    \end{align}
    Observe that $\ip{d\rho}{d\phi_\psi}_h=a_\rho d\phi_\psi(W)=0$. Since $a_\rho=1$ on $\partial M$ and $h$ is $C^2$, $a_\rho-1=O(\rho)$. The Hessians of $\rho$ and $\phi_\psi$ with respect to $h$ are bounded, with the latter controlled by $\norm{\psi}_{C^2}$. Finally, the norm of a $(0,2)$-tensor scales like $|A|_g=\rho^2|A|_h$, and $\phi_\psi=\rho u_\psi$. Applying these facts to \eqref{eq:static-error} proves the estimate.
\end{proof}

\begin{remark}
    A natural question is whether some form of the converse holds. For instance, Huang--Jang--Martin prove that a complete asymptotically hyperbolic manifold in their class is hyperbolic if it admits a nonzero function $f\in C^2_{\mathrm{loc}}(M)$, bounded below, which satisfies
    \[
        \Hess f = f g
    \]
    everywhere \cite[Proposition~4.5]{HuangJangMartin2020}. However, our functions $u_\psi$ are defined only in a collar and satisfy this equation only to leading order on large positive superlevel sets. This comparison does, however, help explain the appearance of the leading hyperbolic term.
\end{remark}

\begin{proposition}\label{prop:totally-convex-complements}
    Fix $\psi$ that is positive somewhere. Then there is a $c_0=c_0(\psi)$ such that for every $c \geq c_0$, setting $\Omega = \Omega_{\psi, c}$, the following holds:
    \begin{enumerate}[label=\textup{(\alph*)}]
        \item $\Hess u_\psi \geq \frac{1}{2} u_\psi g > 0$ on $\Omega$;
        \item $M \setminus \Omega$ is closed and totally convex;
        \item $\partial_\infty \Omega = \overline{\{\psi > 0\}}$, and for $0 < \delta < \max \psi$, the set $\Omega$ contains $\{(\rho,y): \psi(y) \geq \delta, 0<\rho<\delta/c\}$.
    \end{enumerate}
\end{proposition}

\begin{proof}
    Set
    \[
        c_0 := \max\bigl\{ 2\rho_0^{-1}\max \psi, 4C\max \psi, 2\sqrt{C\max\psi \norm{\psi}_{C^2}}\bigr\}.
    \]
    Then $\Omega$ is contained in the collar. On $\Omega$, we have $u_\psi\geq c$ and $\rho\leq(\max\psi)/c$, so Lemma \ref{lem:almost-static-potentials} gives
    \[
        \frac{|E|_g}{u_\psi}\leq C\rho\left(1+\frac{\norm{\psi}_{C^2}}{u_\psi}\right)\leq C\frac{\max\psi}{c}\left(1+\frac{\norm{\psi}_{C^2}}{c}\right)\leq \frac{1}{2}.
    \]
    Thus $\Hess u_\psi\geq \frac{1}{2}u_\psi g$ on $\Omega$, proving (a).

    $\Omega$ is open by construction, and its closure in $M$ is contained in the collar where $u_\psi$ is defined. In particular, $u_\psi=c$ on $\partial\Omega$. Now let $\gamma:[0,1]\to M$ be a nonconstant geodesic segment with endpoints in $M\setminus \Omega$, and suppose by way of contradiction that $\gamma$ leaves $M\setminus \Omega$. If $(a,b)$ is a component of $\gamma^{-1}(\Omega)$, then $(u_\psi \circ \gamma)(a) = (u_\psi \circ \gamma)(b) = c$. On the other hand, by (a), we have
    \[
        (u_\psi\circ\gamma)''=\Hess u_\psi(\dot\gamma,\dot\gamma)\geq \frac{1}{2}(u_\psi\circ\gamma)|\dot\gamma|^2_g>0
    \]
    on $(a,b)$, so strict convexity gives $u_\psi \circ \gamma < c$ there. This contradicts $u_\psi > c$ in $\Omega$. Thus $M\setminus \Omega$ is totally convex, proving (b).

    Let $q \in \partial M$ satisfy $\psi(q) > 0$ and consider the $W$-flow curve $\sigma_q(\rho)=(\rho,q)$. If $0<\rho<\psi(q)/c$, then $u_\psi(\sigma_q(\rho))=\psi(q)/\rho>c$. Letting $\rho\downarrow0$ gives $\{\psi > 0\} \subseteq \partial_\infty \Omega$, and $\overline{\{\psi > 0\}} \subseteq \partial_\infty \Omega$ because $\partial_\infty\Omega$ is closed in $\partial M$.

    Now let $p \in \partial M \setminus \overline{\{\psi > 0\}}$. There exists a neighborhood $U$ of $p$ on which $\psi \leq 0$. Then $u_\psi\leq0$ on $(0,\rho_0) \times U$, so this collar neighborhood misses $\Omega$. Thus $p \notin \partial_\infty \Omega $. Hence $\overline{\{\psi > 0\}} = \partial_\infty \Omega$.

    Finally, if $\psi(y)\geq\delta$ and $0<\rho<\delta/c$, then $u_\psi(\rho,y)>c$. 
\end{proof}

Let $\mathcal{C}$ be the family of all $\Omega_{\psi, c}$ with $0 \leq \psi \in C^\infty(\partial M)$, $\psi \not\equiv 0$, and $c \geq c_0(\psi)$.

\begin{corollary}\label{cor:barriers-in-neighborhoods}
    For every $q\in\partial M$ and every neighborhood $U$ of $q$ in $\overline{M}$, there are $\Omega\in\mathcal{C}$ and a neighborhood $U'$ of $q$ in $\overline{M}$ such that
    \[
        U'\cap M\subseteq\Omega, \qquad
        \overline{\Omega}^{\,\overline{M}}\subseteq U.
    \]
\end{corollary}
\begin{proof}
    In the collar, choose an open set $U_0\subseteq\partial M$ containing $q$ and $0<\rho_1<\rho_0$ so that $[0,\rho_1]\times\overline{U_0}\subseteq U$. Let $\psi$ be a nonnegative smooth function with compact support in $U_0$ and positive on a neighborhood of $q$. Choose $c\geq c_0(\psi)$ large enough that $(\max\psi)/c<\rho_1$, and set $\Omega:=\Omega_{\psi,c}$. Then \eqref{eq:omega-coordinates} gives $\overline{\Omega}^{\,\overline{M}}\subseteq[0,\rho_1)\times\spt\psi\subseteq U$. Choose $0<\delta<\psi(q)$ and a neighborhood $U_q$ of $q$ in $\partial M$ on which $\psi\geq\delta$, and set $U':=[0,\delta/c)\times U_q$. Proposition \ref{prop:totally-convex-complements}(c) gives $U'\cap M\subseteq\Omega$.
\end{proof}

The following is the standard convex function argument which undergirds the convex hull property for stationary varifolds. Anderson uses the same first variation test with the distance, or a smooth approximation to distance, from a totally geodesic hyperplane in his Lemma~5 \cite{Anderson1982}. The only additional point here is that the support is not automatically compact. The condition at conformal infinity supplies the compactness needed for the test field.

\begin{lemma}\label{lem:barrier-exclusion}
    Let $\Omega=\Omega_{\psi,c}$, where $c\geq c_0(\psi)$, and let $V$ be a rectifiable $p$-varifold in $M$, $1\leq p\leq n$, which is stationary in $\Omega$. Suppose that
    \begin{equation}\label{eq:barrier-asymptotic-disjointness}
        \overline{\spt V}^{\,\overline{M}}
        \cap\partial_\infty\Omega=\emptyset.
    \end{equation}
    Then $\spt V\cap\Omega=\emptyset$. In particular, if $T$ is a compactly supported integral $p$-current, $|T|$ is stationary in $M\setminus\spt\partial T$, and $\spt\partial T\cap\Omega=\emptyset$, then $\spt T\cap\Omega=\emptyset$.
\end{lemma}

\begin{proof}
    Let $u=u_\psi$. Since $du(W)=-u/\rho\neq0$ and
    \[
        \Hess u\geq\frac{1}{2}ug>0
    \]
    on $\Omega$, the level sets of $u$ form a strictly convex foliation of $\Omega$. Thus $\nabla u$ is the natural vector field to make admissible, via cutoff functions, for the first variation formula, so as to contradict stationarity.

    Assume by way of contradiction that $\xi\in\spt V\cap\Omega$, and choose $c<a<u(\xi)$. The set
    \[
        K:=\spt V\cap\{u\geq a\}
    \]
    is compactly contained in $\Omega$ by \eqref{eq:barrier-asymptotic-disjointness} and $a>c$. Choose $\chi\in C_c^1(\Omega)$ equal to $1$ on a neighborhood of $K$ and a bounded nondecreasing function $\eta\in C^1(\R)$ such that
    \[
        \eta=0\quad\text{on }(-\infty,a], \qquad
        \eta>0\quad\text{on }(a,\infty).
    \]
    Set $X=\chi\eta(u)\nabla u$, which is admissible for the first variation formula. Observe that at points of $\spt V$ with $u\leq a$, we have $\eta=\eta'=0$, while at points of $\spt V$ with $u>a$, we have $\chi=1$ and $d\chi=0$. Thus for $\mu_V$-almost every tangent $p$-plane $P$,
    \[
        \dvg_P X=\eta'(u)|\nabla^P u|^2+\eta(u)\tr_P \Hess u.
    \]
    Since $\xi\in\spt V\cap\{u>a\}$, the open set $\{u>a\}$ has positive $\mu_V$-measure. Thus by stationarity and Proposition \ref{prop:totally-convex-complements} we reach
    \begin{align*}
        0 = \delta V(X) &= \int_{\{u > a\}} \left(\eta'(u)|\nabla^P u|^2 + \eta(u)\tr_P \Hess u\right)\,d\mu_V \\
        &\geq \frac{p}{2}\int_{\{u > a\}}\eta(u)u\,d\mu_V > 0,
    \end{align*}
    a contradiction.

    For the compactly supported integral current version, compact support implies that the closure of $\spt|T|$ in $\overline{M}$ misses $\partial M$, and the boundary hypothesis makes $|T|$ stationary throughout $\Omega$. The result just proved for stationary varifolds then applies.
\end{proof}

We now intersect the totally convex complements of all $\Omega\in\mathcal{C}$ whose asymptotic boundaries avoid the prescribed closed set, and construct the barrier hull $\mathcal{H}(S)$.

\begin{definition}
    For closed $S\subseteq \partial M$, define
    \[
        \mathcal{H}(S) := \bigcap_{\substack{\Omega \in \mathcal{C} \\ \partial_\infty\Omega \cap S = \emptyset}} (M\setminus \Omega),
    \]
    with the convention that $\mathcal{H}(S):= M$ if it is an empty intersection. We say that $\Omega \in \mathcal{C}$ is \emph{admissible for $S$} if $\partial_\infty \Omega \cap S = \emptyset$.
\end{definition}

Note that the set $\mathcal{H}(S)$ depends on the defining function $\rho$ and the thresholds $c_0(\psi)$. The convex hull $\operatorname{Conv}_{\overline{M}}(S)$ introduced below, however, is canonical once a compactification is fixed.

Since $\mathcal{H}(S)$ is the intersection of closed totally convex sets by Proposition \ref{prop:totally-convex-complements}, it is also closed and totally convex. The next lemma confirms that it has the desired asymptotic boundary.

\begin{lemma}\label{lem:barrier-hull-boundary}
    For every closed $S \subseteq \partial M$,
    \[
        \partial_\infty \mathcal{H}(S) = S.
    \]
\end{lemma}
\begin{proof}
    Let $q \in \partial M\setminus S$. As $S$ is closed, we can choose a neighborhood $U$ of $q$ in $\overline{M}$ such that $U\cap\partial M\subseteq\partial M\setminus S$. By Corollary \ref{cor:barriers-in-neighborhoods}, there are $\Omega\in\mathcal{C}$ and a neighborhood $U'$ of $q$ in $\overline{M}$ such that $U'\cap M\subseteq\Omega$ and $\overline{\Omega}^{\,\overline{M}}\subseteq U$. Then $\Omega$ is admissible for $S$, and hence $\mathcal{H}(S)\cap U'=\emptyset$. It follows that $q\notin\partial_\infty\mathcal{H}(S)$.

    On the other hand, suppose $p \in S$, and let $\sigma_p(\rho)=(\rho,p)$ be the $W$-flow curve converging to $p$ as $\rho\downarrow0$. If $\Omega_{\psi, c}$ is admissible for $S$, then by construction
    \[
        p \notin \overline{\{\psi > 0\}},
    \]
    so $\psi(p) \leq 0$. Then $u_{\psi}(\sigma_p(\rho))=\psi(p)/\rho\leq0<c$, and thus the curve is disjoint from every $\Omega\in\mathcal{C}$ admissible for $S$. It therefore stays in $\mathcal{H}(S)$, forcing $p\in\partial_\infty\mathcal{H}(S)$.
\end{proof}

\begin{proof}[Proof of Theorem \ref{thm:barrier-hull}]
    $\mathcal{H}(S)$ is closed and totally convex by Proposition \ref{prop:totally-convex-complements}, and its asymptotic boundary equals $S$ by Lemma \ref{lem:barrier-hull-boundary}. Suppose that $V$ is stationary and $\partial_\infty\spt V\subseteq S$. For every $\Omega\in\mathcal{C}$ admissible for $S$,
    \[
        \partial_\infty\spt V\cap\partial_\infty\Omega=\emptyset.
    \]
    Lemma \ref{lem:barrier-exclusion} therefore gives $\spt V\cap\Omega=\emptyset$. Intersecting the complements of all such $\Omega$ yields $\spt V\subseteq\mathcal{H}(S)$.
\end{proof}

The following corollary establishes, in the present setting, the convex hull identity suggested by Gromov in \cite[Section~3.2]{Gromov1981}.

\begin{corollary}\label{cor:gromov-hull}
    For a closed set $S\subseteq\partial M$, define
    \[
        \operatorname{Conv}_{\overline{M}}(S):=\bigcap\left\{\overline{K}^{\,\overline{M}}:K\subset M\text{ is closed and totally convex, and }S\subseteq\partial_\infty K\right\}.
    \]
    Then
    \[
        \operatorname{Conv}_{\overline{M}}(S)\cap\partial M=S.
    \]
\end{corollary}

\begin{proof}
    Every member of the intersection contains $S$, so
    \[
        S\subseteq\operatorname{Conv}_{\overline{M}}(S)\cap\partial M.
    \]
    On the other hand, Theorem \ref{thm:barrier-hull} shows that $\mathcal{H}(S)$ is one of the closed totally convex sets in the intersection, with $\partial_\infty\mathcal{H}(S)=S$. Hence
    \[
        \operatorname{Conv}_{\overline{M}}(S)\subseteq\overline{\mathcal{H}(S)}^{\,\overline{M}},
    \]
    which gives the reverse inclusion on $\partial M$.
\end{proof}

\begin{remark}
    Anderson proved the asymptotic boundary identity for the convex hull of an arbitrary closed set in the asymptotic boundary of a complete simply connected Riemannian manifold with pinched negative curvature \cite[Theorem~3.3]{Anderson1983Dirichlet}. Interestingly, an upper negative curvature bound alone does not suffice. Ancona constructed a complete simply connected three-dimensional manifold with sectional curvature $K\leq-1$ and a point $\xi_0$ at infinity such that, for every neighborhood $V$ of $\xi_0$ in the visual compactification, the closed convex hull of $V\cap M$ is the whole manifold \cite[Corollary~C]{Ancona1994}.

    Corollary \ref{cor:gromov-hull} gives the corresponding convex hull identity for every $C^2$ conformally compact asymptotically locally hyperbolic manifold. Taken together with Ancona's example, the result shows that conformal compactness says something about convexity not solely contained in the upper curvature bound.
\end{remark}

In the pinched Cartan--Hadamard setting, Anderson starts with a geodesic sphere and iteratively makes local perturbations along successive spheres---these are the ``scallops.'' Rauch comparison controls the angles of the resulting seams, and this angular control is what enables prescribed boundary control in the full union. The increasing union is convex \cite[Section~2]{Anderson1983Dirichlet}.

The different convex set we find is $K_{\psi,c}:=M\setminus\Omega_{\psi,c}$. Every such $K_{\psi,c}$ contains a fixed compact core, whereas Anderson's construction gives convex domains contained in arbitrarily small truncated cones \cite[Theorem~3.1]{Anderson1983Dirichlet}. Still, the two constructions have the same separation property needed in applications. Namely, if $E\subset M$ is closed and $q\in\partial M\setminus\partial_\infty E$, choose a neighborhood $U$ of $q$ in $\overline{M}$ disjoint from $\overline{E}^{\,\overline{M}}$. Corollary~\ref{cor:barriers-in-neighborhoods} gives $\Omega=\Omega_{\psi,c}\in\mathcal{C}$ and a neighborhood $U'$ of $q$ in $\overline{M}$ such that
\[
    U'\cap M\subseteq\Omega_{\psi,c}, \qquad
    \overline{\Omega_{\psi,c}}^{\,\overline{M}}\subseteq U.
\]
Equivalently, $E\subset K_{\psi,c}$ and $K_{\psi,c}\cap U'=\emptyset$.

In related classes of manifolds, others have considered exactly this kind of convex separation as a defining feature. Ripoll--Telichevesky introduced the following condition on a Cartan--Hadamard manifold with its visual compactification: for every $q$ at infinity and every relative neighborhood $\mathcal{U}$ of $q$, there is a $C^2$ open set $\Omega\subset M$ such that
\[
    q\in\operatorname{Int}_{\partial_\infty M}(\partial_\infty\Omega)\subset\mathcal{U}, \qquad
    M\setminus\Omega\ \text{is geodesically convex}.
\]
They call this the strict convexity condition, or SC condition \cite[Definition~1.1]{RipollTelichevesky2015}. The same definition makes sense for a conformal compactification without the Cartan--Hadamard hypothesis, so here we call it the \emph{conformal SC condition}, to keep the two classes distinct.

Choi's convex conic neighborhood condition asks that for each $x$ in the visual boundary and every $y\neq x$, there are disjoint open neighborhoods $V_x,V_y$ of $x,y$, respectively, in the cone topology such that $V_x\cap M$ is convex with $C^2$ boundary \cite[Definition~4.6]{Choi1984}. The two conditions are similar, but the roles of the set and its complement are reversed, for in the SC condition, the small neighborhood at infinity is the part removed, while its complement is convex. One might ask: are these properties equivalent? Ripoll--Telichevesky observed in 2015 that it remained open whether the two conditions are equivalent \cite{RipollTelichevesky2015}.

The boundary of each $\Omega_{\psi,c}$ is $C^2$ because $du_\psi(W)=-u_\psi/\rho\neq0$ on its level set. Corollary \ref{cor:barriers-in-neighborhoods} therefore shows that every $C^2$ conformally compact asymptotically locally hyperbolic manifold satisfies the conformal SC condition. If, in addition, $(M,g)$ is Cartan--Hadamard and the identity on $M$ extends to a homeomorphism from its visual compactification to $\overline{M}$, then the notions of asymptotic boundary determined by the visual and conformal compactifications agree, so the same $\Omega_{\psi,c}$ is an SC set in the original sense.

The existing SC theory has strong consequences under its own global hypotheses. Ripoll--Telichevesky assume throughout that $M$ is Hadamard with $K_M\leq-k^2<0$. Under structural conditions on an operator in divergence form, SC implies regularity at infinity \cite[Theorem~2.5]{RipollTelichevesky2015}. Combined with solvability on an exhaustion by bounded domains and compactness for uniformly bounded sequences of solutions on relatively compact subsets, this gives solvability of the asymptotic Dirichlet problem \cite[Theorem~2.6]{RipollTelichevesky2015}. In particular, their Theorem~2.8 treats the minimal graph equation and the $p$-Laplacian, with uniqueness in the corresponding classes.

Casteras--Holopainen--Ripoll use SC in asymptotic Plateau results \cite[Theorems~1.5--1.8]{CHR2018}. Under SC alone, their theorem in arbitrary codimension uses coefficients modulo $2$, and their conclusion with integer multiplicity is codimension one with boundary data arising as the boundary of a set. Their integer results in arbitrary codimension require additional rotational symmetry, or curvature and integrability assumptions. An additional distinction to be made is that those results establish the asymptotic boundary of the support, but do not conclude the literal current boundary identity in the compactification or the absence of boundary mass proved here.

To summarize the contrast, the global hypotheses of those results are not assumed here. The compact part may have positive curvature or nontrivial topology, so $M$ is not necessarily Cartan--Hadamard and its conformal boundary is not automatically a visual boundary. Moreover, SC is a barrier hypothesis and does not by itself give solvability on compact domains or the interior estimates required for a particular equation. Still, the aforementioned common features indicate that solvability of the general asymptotic Dirichlet problem in the asymptotically locally hyperbolic setting, including the minimal graph equation and the $p$-Laplacian, should be eminently achievable using the family $\mathcal C$ once the usual existence theory on compact domains and the requisite interior estimates are in place.

\section{Compact Plateau problems and mass estimates}\label{sec:compact-plateau-and-mass-estimates}

Anderson's hyperbolic construction carries the prescribed boundary cycle to large geodesic spheres and fills the resulting cycles by compact Plateau minimizers. Comparison with cones gives an upper mass bound, while a homological argument forces the minimizers to meet a fixed compact set, after which geodesic ball monotonicity supplies the lower bound used to prevent disappearance in the limit \cite[Theorem~1 and the proof of Theorem~3]{Anderson1982}. A general asymptotically locally hyperbolic manifold lacks a global pole and the curvature bounds needed for that kind of comparison. Here, radial cylinders, which play well with currents, replace the cones, and a weighted first variation estimate in the collar carries their mass bound inward. The identity $\partial\overline{T}=A$ in the compactification will, instead, be the mechanism which ultimately prevents the limit of our minimizing sequence from disappearing.

For the rest of the proof, fix $p$, $A$, and $Q$ from Theorem \ref{thm:asymptotic-plateau}, set $S:=\spt A$, and let $q:=p-1\geq1$. We now need $p\geq2$ for the mass estimates below, because when $p=1$, the radial cylinder we consider has mass of order $t$, but with $q=0$, the collar estimate gives no decay to make up for this growth. The case $p=1$ thus requires a separate argument.

Recall that we let $O_t = \{r\leq t\}$ and $N_t = \partial O_t = \{r= t\}$. Let $G_r := e^{2r}g_r$ denote the metric induced on each $N_r$ by $g$. Then
\[
    \Hess r|_{TN_r} = \frac{1}{2} \partial_r G_r = e^{2r}\left(g_r + \frac{1}{2}\partial_r g_r\right), \qquad
    \Hess r(\partial_r, \cdot) = 0,
\]
so if we take $a > r_0$ big enough, then
\[
    \Hess r(v,v) \geq (1-C_0e^{-r})|v|_g^2,
\]
for $r \geq a$ and $v \in TN_r$, with
\begin{equation}\label{eq:collar-smallness}
    pC_0e^{-a} \leq 1.
\end{equation}

For $t \geq a$, define the following:
\[
    f_t:\overline{M}\to O_t, \qquad
    f_t = \text{Id on } O_t,\qquad
    f_t(r,\theta) = (t,\theta) \text{ for } r\geq t.
\]
So the $f_t$ retract the collar to $N_t$. In collar coordinates,
\[
    f_t(x,\theta) = (\max\{x,e^{-t}\}, \theta).
\]

Using an ambient collar of $\partial M$, each $f_t$ extends to a Lipschitz map from a neighborhood of $\overline{M}$ in $\widetilde{M}$ into $M$, with $f_t(\overline{M})\subseteq O_t$. Hence the pushforward is defined for all currents supported in $\overline{M}$ and, in particular, for the filling $Q$. Viewed as a map on $(M,g)$, we can show that $f_t$ is in fact $1$-Lipschitz.

\begin{lemma}\label{lem:radial-retraction}
    For every $t \geq a$, $f_t: (M,g) \to (O_t,g)$ is $1$-Lipschitz and fixes $N_t$.
\end{lemma}
\begin{proof}
    The latter conclusion is clear. Since $\partial_r G_r=2\Hess r|_{TN_r}$ is positive definite for $r\geq a$, we have $G_t\leq G_r$ whenever $r\geq t\geq a$. Hence for $r>t$
    \[
        |df_t(b\partial_r+v)|_g^2=|v|_{G_t}^2\leq |v|_{G_r}^2\leq |b\partial_r+v|_g^2.
    \]
    On $O_t$ the map is the identity. The estimate holds a.e.\ across $N_t$, which suffices to show that $f_t$ is globally $1$-Lipschitz.
\end{proof}

Next we bring the prescribed cycle into the interior. Let $i_t(\theta):=(t,\theta)$ and define
\[
    A_t := (i_t)_\# A.
\]
Then we have
\[
    \Mass_{g}(A_t) = e^{qt}\Mass_{g_t}(A),\qquad
    \sup_{t \geq a} \Mass_{g_t}(A) < \infty.
\]
Fix $Q_a:=(f_a)_\#Q$, so that $\partial Q_a=A_a$ and $\spt Q_a\subseteq O_a$. For $t\geq a$, let
\[
    H:[a,t]\times\partial M\longrightarrow M, \qquad
    H(s,\theta):=(s,\theta),
\]
and define
\[
    C_{a,t}:=H_\#\bigl(\llbracket[a,t]\rrbracket\times A\bigr).
\]
As $\partial A=0$, the homotopy formula gives
\[
    \partial C_{a,t}=A_t-A_a.
\]
Since $g=dr^2+G_r$ and $dr\perp TN_r$, the area formula gives
\[
    \Mass_g(C_{a,t})\leq\int_a^t\Mass_g(A_s)\,ds=\int_a^t e^{qs}\Mass_{g_s}(A)\,ds\leq Ce^{qt}.
\]
Thus
\[
    P_t:=Q_a+C_{a,t}
\]
is supported in $O_t$ and satisfies $\partial P_t=A_t$. After increasing $C$ to include the fixed mass of $Q_a$,
\[
    \Mass_g(P_t)\leq\Mass_g(Q_a)+\Mass_g(C_{a,t})\leq Ce^{qt},
\]
with $C$ independent of $t$.

\begin{proposition}\label{prop:compact-plateau}
    For every $t\geq a$, there exists $T_t \in \I_p(M)$ with
    \[
        \partial T_t = A_t, \qquad
        \spt T_t \subseteq O_t,
    \]
    and least mass among currents supported in $O_t$ with boundary $A_t$. There is also a constant $C$, independent of $t$, such that
    \begin{equation}\label{eq:compact-plateau-upper-bound}
        \Mass_g(T_t)\leq Ce^{qt}.
    \end{equation}
    Additionally:
    \begin{enumerate}[label=\textup{(\alph*)}]
        \item $T_t$ has least mass among all compactly supported Plateau fillings of $A_t$ in all of $M$;
        \item $T_t$ is locally absolutely area minimizing, and $|T_t|$ is stationary in $M\setminus \spt A_t$;
        \item for every $\Omega_{\psi,c}$ admissible for $S$, there is $t_0=t_0(\psi,c)$ such that $\spt T_t\cap\Omega_{\psi,c}=\emptyset$ whenever $t\geq t_0$.
    \end{enumerate}
\end{proposition}

\begin{proof}
    Since $P_t$ is a competitor, Federer--Fleming compactness and lower semicontinuity of mass give a minimizer supported in the compact set $O_t$ \cite{FedererFleming1960}. That it is minimizing and the estimate for $P_t$ give \eqref{eq:compact-plateau-upper-bound}.

    If $X$ is any compactly supported filling of $A_t$ in $M$, then $(f_t)_{\#}X$ is supported in $O_t$ and has boundary $A_t$. Lemma \ref{lem:radial-retraction} therefore gives
    \[
        \Mass_{g}(T_t) \leq \Mass_g((f_t)_{\#}X) \leq \Mass_g(X).
    \]
    Thus $T_t$ minimizes among all compactly supported fillings of $A_t$ in $M$. If $Y-T_t$ has compact support and $\partial Y=\partial T_t$, then $Y$ is also compactly supported, and hence
    \[
        \Mass_g(T_t)\leq\Mass_g(Y).
    \]
    If $W\Subset M$ contains $\spt(Y-T_t)$, then $T_t=Y$ on $M\setminus W$, so their mass measures agree there. Therefore
    \begin{align*}
        \norm{T_t}_g(W) + \norm{T_t}_g(M\setminus W) &\leq \norm{Y}_g(W) + \norm{Y}_g(M\setminus W), \\
        \norm{T_t}_g(M\setminus W) &= \norm{Y}_g(M\setminus W).
    \end{align*}
    Subtracting gives $\norm{T_t}_g(W)\leq\norm{Y}_g(W)$. Thus $T_t$ is locally absolutely area minimizing, and the standard first variation argument gives stationarity in $M\setminus\spt A_t$.

    Finally, fix $\Omega_{\psi,c}$ admissible for $S$. Its closure in $\overline{M}$ is disjoint from the compact set $S\subset\partial M$. Since $i_t|_S$ converges uniformly to inclusion on the boundary, $\spt A_t$ misses $\Omega_{\psi,c}$ for all sufficiently large $t$. And an application of Lemma \ref{lem:barrier-exclusion} supplies the desired exclusion.
\end{proof}

We next control the portion of $T_t$ between $N_a$ and a fixed level $N_s$ below $N_t$. Proposition \ref{prop:compact-plateau}(b) gives stationarity in this region, while the right choice of $w$ compensates for the $O(e^{-r})$ error in $\Hess r$. 

\begin{lemma}\label{lem:weighted-collar-monotonicity}
    Let
    \[
        \delta(r) := C_0 e^{-r}, \qquad
        w(r) := \exp(-qC_0e^{-r}),
    \]
    and let $Z$ be an integral $p$-current of finite mass such that its associated varifold $|Z|$ is stationary in $\{a<r<t\}$. Define
    \[
        B(s):= \int_{\{a < r < s\}} w(r)\,d\norm{Z}_g.
    \]
    Then $s\mapsto e^{-qs}B(s)$ is nondecreasing on $(a,t)$. Equivalently, for every $a<s<u<t$,
    \[
        B(s) \leq e^{-q(u-s)}B(u).
    \]
    If $\Mass_g(Z) \leq Le^{qt}$, then
    \[
        \norm{Z}_g(\{a < r < s\}) \leq CLe^{qs}, \qquad
        a<s\leq t,
    \]
    where $C=w(a)^{-1}\leq e$.
\end{lemma}
\begin{proof}
    Write $\mu=\norm{Z}_g$. For $\mu$-a.e.\ $y$, let $P_y$ be the approximate tangent $p$-plane of $Z$ and let
    \[
        \alpha(y)=|\nabla^{P_y} r|^2\in[0,1].
    \]
    Since $\Hess r(\partial_r, \cdot)$ vanishes, the Hessian estimate in the collar gives
    \[
        \tr_{P_y} \Hess r\geq(1-\delta)(p-\alpha(y)).
    \]
    We have by construction $w'=q\delta w$. Hence, using $q=p-1$,
    \begin{align}
        \dvg_{P_y}(w\nabla r) &= w'\alpha(y) + w\tr_{P_y} \Hess r \notag \\
        &\geq w(q\delta\alpha(y) + (1 - \delta)(p - \alpha(y))) \notag \\
        &= w(q + (1 - p\delta)(1 - \alpha(y))) \geq qw, \label{eq:weighted-divergence}
    \end{align}
    where the last inequality follows from \eqref{eq:collar-smallness}.

    Now we use a standard smooth cutoff argument for stationary varifolds. Fix $a<s<u<t$. Choose $\chi_\kappa\in C^\infty(\R)$ nondecreasing such that
    \[
        \chi_\kappa=0\quad\text{on }\{r\leq a+\kappa\}, \qquad
        \chi_\kappa=1\quad\text{on }\{r\geq a+2\kappa\},
    \]
    where $0<2\kappa<s-a$, and choose a nondecreasing $\eta\in C^\infty(\R)$ with $\eta=0$ on $(-\infty,0]$ and $\eta=1$ on $[1,\infty)$. For $\sigma\in[s,u]$ and $\varepsilon>0$, set
    \[
        F_{\kappa,\varepsilon}(\sigma):=\int w(r)\chi_\kappa(r)\eta\!\left(\frac{\sigma-r}{\varepsilon}\right)\,d\mu.
    \]
    The vector field
    \[
        X_\sigma:=w(r)\chi_\kappa(r)\eta\!\left(\frac{\sigma-r}{\varepsilon}\right)\nabla r
    \]
    is compactly supported in $\{a<r<t\}$. Stationarity, \eqref{eq:weighted-divergence}, and $\chi_\kappa'\geq0$ give
    \begin{align*}
        qF_{\kappa,\varepsilon}(\sigma) &\leq \frac{1}{\varepsilon}\int w\chi_\kappa\alpha\eta'\!\left(\frac{\sigma - r}{\varepsilon}\right)\,d\mu \\
        &\leq \frac{1}{\varepsilon}\int w\chi_\kappa\eta'\!\left(\frac{\sigma - r}{\varepsilon}\right)\,d\mu = F_{\kappa,\varepsilon}'(\sigma),
    \end{align*}
    where the second inequality uses $0\leq\alpha\leq1$, and the equality follows by differentiation under the integral. Thus
    \[
        \frac{d}{d\sigma}\bigl(e^{-q\sigma}F_{\kappa,\varepsilon}(\sigma)\bigr)\geq0,
    \]
    and so
    \[
        F_{\kappa,\varepsilon}(s)\leq e^{-q(u-s)}F_{\kappa,\varepsilon}(u).
    \]
    Letting first $\varepsilon\downarrow 0$, and then $\kappa\downarrow 0$, dominated convergence supplies
    \[
        B(s)\leq e^{-q(u-s)}B(u).
    \]
    The choice $\eta(0)=0$ gives the strict sublevels. If $\Mass_g(Z)\leq Le^{qt}$, then
    \[
        B(s)\leq e^{-q(u-s)}B(u)\leq Le^{q(t-u+s)}, \qquad
        s<u<t.
    \]
    Letting $u\to t$ and using $w(r)\geq w(a)$ for $r\geq a$ gives
    \[
        \norm{Z}_g(\{a<r<s\})\leq w(a)^{-1}Le^{qs}.
    \]
    For $s=t$, the conclusion follows directly from the assumed total mass bound.
\end{proof}

\begin{remark}
    When $g$ is actually hyperbolic, take $r=-\log z$ in the upper half-space model. Then
    \[
        \Hess r=g-dr\otimes dr,
    \]
    so one may take $w\equiv1$, and the preceding argument gives
    \[
        e^{-qs}\norm{Z}_g(\{a<r<s\})
    \]
    nondecreasing. This is the horospherical analogue of the geodesic ball monotonicity used by Anderson \cite[Theorem~1]{Anderson1982}. 
\end{remark}

Since $\spt A_t\subseteq N_t$, Proposition \ref{prop:compact-plateau} gives stationarity of $|T_t|$ in $\{a<r<t\}$. Combining Lemma \ref{lem:weighted-collar-monotonicity} with \eqref{eq:compact-plateau-upper-bound} yields the key estimate
\begin{equation}\label{eq:uniform-collar-mass}
    \norm{T_t}_g(\{a < r < s\}) \leq Ce^{qs}
\end{equation}
for $a < s \leq t$, where $C$ does not depend on $t$.

Then \eqref{eq:uniform-collar-mass} controls the part of $T_t$ in the collar but does not yet say anything about its mass in $O_a$. So we next choose a slice in the fixed annulus $\{a<r<a+1\}$, fill it in $O_{a+1}$, and apply minimality of $T_t$ to compare this filling with the part of $T_t$ inside the slice. 

\begin{proposition}\label{prop:uniform-g-mass}
    There is $C$ such that, for every $t\geq a$ and every $a \leq s \leq t$,
    \[
        \norm{T_t}_g(O_s) \leq Ce^{qs}.
    \]
\end{proposition}

\begin{proof}
    Suppose first that $t>a+1$. By the slicing inequality and \eqref{eq:uniform-collar-mass},
    \[
        \int_a^{a+1}\Mass_g(\langle T_t,r,s\rangle)\,ds\leq\norm{T_t}_g(\{a<r<a+1\})\leq C.
    \]
    For a.e.\ $s$, the slice is integral and $\norm{T_t}_g(N_s)=0$. We can then choose $\tau_t\in(a,a+1)$ such that
    \[
        Z_t:=\langle T_t,r,\tau_t\rangle\in\I_q(M), \qquad
        \norm{T_t}_g(N_{\tau_t})=0, \qquad
        \Mass_g(Z_t)\leq C.
    \]
    Since $\tau_t<t$ and $\partial T_t$ is supported on $N_t$, the slice satisfies
    \[
        Z_t=\partial\bigl( T_t\mres\{r< \tau_t\}\bigr).
    \]
    The current on the right is supported in $O_{\tau_t}\subseteq O_{a+1}$, so $Z_t$ is an integral boundary there. By Federer--Fleming compactness, the integral $q$-cycles in the zero homology class of compact Lipschitz neighborhood retract $O_{a+1}$, with mass at most $C$, form a compact family in the flat topology \cite[Corollary~7.3 and Theorem~9.5]{FedererFleming1960}. The least mass of a filling in $O_{a+1}$ is continuous on this family \cite[Theorem~9.11 and Remark~9.17]{FedererFleming1960}, and is thus uniformly bounded. Hence there is a current $F_t\in\I_p(M)$ such that
    \[
        \partial F_t=Z_t, \qquad
        \spt F_t\subseteq O_{a+1}, \qquad
        \Mass_g(F_t)\leq C.
    \]
    It follows that
    \[
        \partial\bigl( T_t\mres\{r> \tau_t\}\bigr)=A_t-Z_t
    \]
    and hence
    \[
        Y_t:=T_t\mres\{r>\tau_t\}+F_t
    \]
    is a compactly supported filling of $A_t$. Proposition \ref{prop:compact-plateau}(a), together with $\norm{T_t}_g(N_{\tau_t})=0$, gives
    \begin{align*}
        \norm{T_t}_g(\{r < \tau_t\}) + \norm{T_t}_g(\{r > \tau_t\}) &\leq \Mass_g(Y_t) \\
        &\leq \norm{T_t}_g(\{r > \tau_t\}) + \Mass_g(F_t).
    \end{align*}
    Thus
    \[
        \norm{T_t}_g(\{r<\tau_t\})\leq C.
    \]
    Now let $a\leq s<t$ and choose $0<\varepsilon<t-s$. Since $\tau_t>a$,
    \[
        O_s\subseteq\{r<\tau_t\}\cup\{a<r<s+\varepsilon\}.
    \]
    The preceding estimate and \eqref{eq:uniform-collar-mass} imply
    \[
        \norm{T_t}_g(O_s)\leq C+Ce^{q(s+\varepsilon)}.
    \]
    Letting $\varepsilon\downarrow0$ gives the estimate. When $s=t$, the estimate follows from \eqref{eq:compact-plateau-upper-bound}. If $a\leq t\leq a+1$ and $s\leq t$, then
    \[
        \norm{T_t}_g(O_s)\leq\Mass_g(T_t)\leq Ce^{qt}\leq Ce^qe^{qs},
    \]
    which proves the claim.
\end{proof}

Proposition \ref{prop:uniform-g-mass} gives growth of order $e^{qs}$ for the mass in the metric $g$ on $O_s$. As already noted, compactification multiplies the mass measure of a $p$-current by $e^{-pr}$. Since $p-q=1$, summing over unit annuli leaves a factor $e^{-R}$, and we can get the following uniform estimates.

\begin{lemma}\label{lem:uniform-mass-bounds}
    There is $C$, independent of $t\geq a$, such that
    \[
        \Mass_{\widetilde{g}}(T_t)+\Mass_{\widetilde{g}}(\partial T_t)\leq C.
    \]
    Moreover, for every $R\geq a$,
    \begin{equation}\label{eq:compact-mass-decay}
        \norm{T_t}_{\overline{g}}(\{r>R\})\leq Ce^{-R}.
    \end{equation}
\end{lemma}

\begin{proof}
    On the collar, conformal scaling immediately gives
    \[
        d\norm{T_t}_{\overline{g}}=e^{-pr}\,d\norm{T_t}_g.
    \]
    For $j\geq0$, set
    \[
        E_j:=\{R+j<r\leq R+j+1\}.
    \]
    Proposition \ref{prop:uniform-g-mass}, together with $\spt T_t\subseteq O_t$, gives, for every $j\geq0$,
    \[
        \norm{T_t}_g(E_j)\leq\norm{T_t}_g(O_{R+j+1})\leq Ce^{q(R+j+1)}.
    \]
    Therefore
    \begin{align*}
        \norm{T_t}_{\overline{g}}(\{r > R\}) &\leq \sum_{j=0}^\infty e^{-p(R + j)}\norm{T_t}_g(E_j) \\
        &\leq C\sum_{j=0}^\infty e^{-p(R + j)}e^{q(R + j + 1)} \leq Ce^{-R},
    \end{align*}
    since $p-q=1$. This proves \eqref{eq:compact-mass-decay}. On $O_a$, the metrics $g$, $\overline{g}$, and $\widetilde{g}$ are uniformly equivalent, so the $g$-mass bound there together with the decay estimate gives $\Mass_{\widetilde{g}}(T_t)\leq C$.

    Finally, the metric induced by $\overline{g}$ on $N_t$ is $g_t$, and thus
    \[
        \Mass_{\overline{g}}(A_t)=\Mass_{g_t}(A)\leq C.
    \]
    Uniform equivalence of $\widetilde{g}$ and $\overline{g}$ on $\overline{M}$ proves the estimate for the boundary mass.
\end{proof}

\section{Passage to the limit}\label{sec:passage-to-limit}

We now finish the proof of Theorem \ref{thm:asymptotic-plateau} by taking a limit of the compact Plateau minimizers $T_t$. The estimates from the previous section give the needed compactness and rule out mass accumulation on $\partial M$. Once we show that $A_t\to A$ in the flat norm, the desired conclusions follow by passing to the limit.

\begin{lemma}\label{lem:boundary-cycle-convergence}
    For every $t\geq a$, there is $D_t\in\I_p(\overline{M})$ such that
    \[
        \partial D_t=A_t-A, \qquad
        \Mass_{\widetilde{g}}(D_t)\leq Ce^{-t}.
    \]
    In particular,
    \[
        \mathcal{F}_{\widetilde{g}}(A_t-A)\leq Ce^{-t}.
    \]
\end{lemma}

\begin{proof}
    Let
    \[
        \iota:[0,\varepsilon_0]\times\partial M\longrightarrow\overline{M}
    \]
    denote the collar parametrization, where the first coordinate is $x$, and set
    \[
        D_t:=\iota_\#\bigl(\llbracket[0,e^{-t}]\rrbracket\times A\bigr).
    \]
    Since $\partial A=0$, $\partial D_t=A_t-A$. Uniform equivalence of the collar metrics gives
    \[
        \Mass_{\widetilde{g}}(D_t)\leq Ce^{-t}\Mass_{h_0}(A),
    \]
    and the flat norm estimate follows.
\end{proof}

Through Lemma \ref{lem:uniform-mass-bounds} we now have the bounds required for compactness and uniform decay of the $\overline{g}$-mass, while the convergence just proved will control the boundary of any limit.

\begin{proposition}\label{prop:limit-current}
    There is a sequence $t_j\to\infty$ and a $\overline{T}\in\I_p(\overline{M})$ such that
    \[
        \mathcal{F}_{\widetilde{g}}(T_{t_j}-\overline{T})\longrightarrow0, \qquad
        \partial\overline{T}=A, \qquad
        \norm{\overline{T}}_{\overline{g}}(\partial M)=0.
    \]
\end{proposition}

\begin{proof}
    Lemma \ref{lem:uniform-mass-bounds} and Federer--Fleming compactness on the closed manifold $(\widetilde{M},\widetilde{g})$ give a subsequence converging in the ambient flat norm to an integral current $\overline{T}$. Since every $T_{t_j}$ is supported in the closed subset $\overline{M}\subset\widetilde{M}$, the limit, too, is supported there. The boundary operator is continuous in the flat norm, and so Lemma \ref{lem:boundary-cycle-convergence} gives
    \[
        \partial\overline{T}=A.
    \]

    What remains, then, is to rule out any mass accumulation at the boundary. Choose decreasing open neighborhoods $U_R\subset\widetilde{M}$ of $\partial M$ such that $\bigcap_RU_R=\partial M$ and $U_R\cap M\subset\{r>R\}$. Lower semicontinuity on open sets, uniform equivalence of $\widetilde{g}$ and $\overline{g}$, and \eqref{eq:compact-mass-decay} imply
    \[
        \norm{\overline{T}}_{\widetilde{g}}(U_R)\leq\liminf_{j\to\infty}\norm{T_{t_j}}_{\widetilde{g}}(U_R)\leq Ce^{-R}.
    \]
    Continuity from above of the mass measure gives $\norm{\overline{T}}_{\widetilde{g}}(\partial M)=0$. The identity for $\overline{g}$ follows.
\end{proof}

\begin{proof}[Proof of Theorem \ref{thm:asymptotic-plateau}]
    Let $\overline{T}$ be the current in Proposition \ref{prop:limit-current} and $T=\overline{T}\mres M$, and write $S=\spt A$. For every test $(p-1)$-form $\omega$ that is compactly supported in $M$,
    \[
        \partial T(\omega)=\partial\overline{T}(\omega)=0,
    \]
    because $\partial\overline{T}$ is supported on $\partial M$.

    Let $W\Subset M$. For all sufficiently large $j$, the currents $T_{t_j}$ have no boundary near $W$, are absolutely area minimizing there by Proposition \ref{prop:compact-plateau}, and have uniformly bounded $g$-mass on compact subsets by Proposition~\ref{prop:uniform-g-mass}. Flat convergence in the smooth ambient manifold implies weak convergence of currents locally in $M$. Thus the closure theorem for minimizing currents \cite{Simon2016} shows that $T$ is locally absolutely area minimizing in $(M,g)$.

    Fix an $\Omega_{\psi,c}$ that is admissible for $S$. Proposition \ref{prop:compact-plateau}(c) shows that, for all sufficiently large $j$, the current $T_{t_j}$ misses $\Omega_{\psi,c}$, and so the weak limit $T$ also misses it. Intersecting over all such $\Omega_{\psi,c}$ gives $\spt T\subseteq\mathcal{H}(S)$. Lemma \ref{lem:barrier-hull-boundary} then gives
    \begin{equation}\label{eq:asymptotic-support-inclusion}
        \partial_\infty\spt T
        \subseteq\partial_\infty\mathcal{H}(S)
        =S.
    \end{equation}
    The absence of boundary mass implies that, as a current in $\widetilde{M}$, $\overline{T}$ is the extension by zero of its restriction $T=\overline{T}\mres M$, and that the support is
    \begin{equation}\label{eq:support-closure-identity}
        \spt\overline{T}=\overline{\spt T}^{\,\overline{M}}.
    \end{equation}
    Indeed, $\overline{\spt T}^{\,\overline{M}}\subseteq\spt\overline{T}$ is immediate because $\spt\overline{T}$ is closed and contains $\spt T$. For the other inclusion, notice that points of $\spt\overline{T}\cap M$ lie in $\spt T$, while every neighborhood of a point in $\spt\overline{T}\cap\partial M$ has positive $\overline{T}$-mass away from $\partial M$, and hence meets $\spt T$. This gives \eqref{eq:support-closure-identity}. 
    
    Then
    \[
        S=\spt A=\spt(\partial\overline{T})\subseteq\spt\overline{T}.
    \]
    Combining $S\subseteq\spt\overline{T}$ with \eqref{eq:support-closure-identity} proves $S\subseteq\partial_\infty\spt T$. Thus with \eqref{eq:asymptotic-support-inclusion} we have $\partial_\infty\spt T=S$.

    Finally, if $T=0$, then $\overline{T}=0$, contradicting $\partial\overline{T}=A\neq0$. Hence $T\neq0$, and we are done.
\end{proof}

Beyond existence, the $g$-mass bound and the decay of the $\overline{g}$-mass also pass to the limiting current.

\begin{corollary}
    There is a constant $C$ such that the current furnished by Theorem \ref{thm:asymptotic-plateau} satisfies
    \[
        \norm{T}_g(O_R)\leq Ce^{(p-1)R}, \qquad
        \norm{\overline{T}}_{\overline{g}}(\{r>R\})\leq Ce^{-R}
    \]
    for every $R\geq a$.
\end{corollary}

\begin{proof}
    Fix $R\geq a$. For $\varepsilon>0$, local lower semicontinuity and Proposition \ref{prop:uniform-g-mass} give
    \[
        \norm{T}_g(O_R)\leq\norm{T}_g(\{r<R+\varepsilon\})\leq\liminf_{j\to\infty}\norm{T_{t_j}}_g(\{r<R+\varepsilon\})\leq Ce^{(p-1)(R+\varepsilon)}.
    \]
    Letting $\varepsilon\downarrow0$ proves the first estimate. For $L>R$, lower semicontinuity on the relatively compact region $\{R<r<L\}$ and \eqref{eq:compact-mass-decay} give
    \[
        \norm{\overline{T}}_{\overline{g}}(\{R<r<L\})\leq Ce^{-R}.
    \]
    Let $L\uparrow\infty$ and use $\norm{\overline{T}}_{\overline{g}}(\partial M)=0$.
\end{proof}

\begin{proof}[Proof of Corollary \ref{cor:codim-one}]
    Here, the current dimension is $n$, and the ambient dimension is $n+1$. The codimension one regularity theorem for locally area minimizing integral currents gives the description of the regular and singular sets, while the constancy theorem gives the multiplicity statement; see \cite{Simon2016}. When $n\leq6$, the support is a closed smooth embedded hypersurface in $M$, so the embedding is proper. Conformal compactness makes the metric $g$ complete, and a properly immersed submanifold of a complete Riemannian manifold is, of course, complete in its induced metric.
\end{proof}

\begin{remark}\label{rem:fixed-slice}
    If a smooth spacelike slice has an induced asymptotically locally hyperbolic metric satisfying the hypotheses of Theorem \ref{thm:asymptotic-plateau}, then, in the hypersurface case, that theorem constructs anchored integral currents which are locally absolutely area minimizing under compactly supported variations on the fixed slice. In AdS/CFT, Wall's maximin construction includes a minimizing step on a fixed slice in which one minimizes area among surfaces anchored at the entangling surface and homologous to the boundary region, and in quantum maximin, generalized entropy replaces area in this step \cite{Wall2014,AkersEngelhardtPeningtonUsatyuk2020}. In their discussion of existence on a fixed Cauchy slice, Akers et al.\ write that, even for minimal area surfaces, they are unaware of a fully mathematically rigorous proof that the surface is well defined in the limit as the boundary cutoff is taken to infinity \cite[Section~3.3]{AkersEngelhardtPeningtonUsatyuk2020}. Theorem \ref{thm:asymptotic-plateau} supplies a rigorous limit for the classical anchored Plateau problem considered here, which is closely related to the holographic problem.
\end{remark}

\bibliographystyle{amsplain}
\bibliography{refs}

\end{document}